\documentclass[12pt]{article}
\usepackage[english]{babel}
\usepackage{amsfonts}
\usepackage{extsizes}
\usepackage[utf8]{inputenc}
\usepackage{amsthm}
\usepackage{ mathrsfs }
\usepackage{ amsmath }
\usepackage{ amssymb }
\usepackage{enumerate}
\usepackage{cite}
\usepackage{bbm}
\usepackage{makeidx}
\usepackage{graphicx}
\usepackage{frontespizio}
\usepackage{color}

\theoremstyle{definition}
\newtheorem{definition}{Definition}

\newtheorem{remark}{Remark}
\newtheorem{theorem}{Theorem}
\newtheorem{lemma}{Lemma}

\DeclareMathOperator\supp{supp}

\DeclareMathOperator\rr{\mathbb{R}}
\DeclareMathOperator\nn{\mathbb{N}}

\newcommand\addtag{\refstepcounter{equation}\tag{\theequation}}

\newcommand{\Hmm}[1]{\leavevmode{\marginpar{\tiny%
$\hbox to 0mm{\hspace*{-0.5mm}$\leftarrow$\hss}%
\vcenter{\vrule depth 0.1mm height 0.1mm width \the\marginparwidth}%
\hbox to
0mm{\hss$\rightarrow$\hspace*{-0.5mm}}$\\\relax\raggedright #1}}}

\begin{document}

\title{On Stein's extension operator preserving  Sobolev-Morrey spaces }

\author{Pier Domenico Lamberti and Ivan Yuri Violo}

\date{}

\maketitle

\begin{abstract}  We prove that Stein's Extension Operator preserves Sobolev-Mor\-rey spaces, that is  spaces of functions with 
	weak derivatives in Morrey spaces. The analysis concerns classical and ge\-ne\-ra\-li\-zed  Mor\-rey spaces on bounded and unbounded domains with Lipschitz boundaries in the $n$-dimensional  Euclidean space.  
	\end{abstract}

{\bf Keywords}: Extension operator, Lipschitz domains, So\-bo\-lev and Mor\-rey spa\-ces\\
{\bf 2010 Mathematics Subject Classification}: 46E35,  46E30, 42B35\\

\section{Introduction}

One of the fundamental tools in the theory of Sobolev Spaces and their applications to partial differential equations is Stein's  Extension operator
which allows to extend functions defined on a Lipschitz domain (i.e., a connected open set with Lipschitz continuous boundary)   $\Omega \subset \rr^n $, $n\geq 2$,  to the whole ambient space  $\rr^n$ preserving smoothness and summability. 
Namely, in 1967  E. Stein~\cite{steinorsay} defined a linear continuous  operator
$T$ from the  Sobolev space $W^{l,p}(\Omega )$ to  the Sobolev space $W^{l,p}(\rr^n )$ such that $Tf_{_{|\Omega} } =f$  for all $f\in W^{l,p}(\Omega ) $, see also \cite{stein}.
It is important to observe that  Stein's  Extension operator is universal in the sense that   the definition of $Tf$ is given by means of a formula which is independent of $l\in \nn$ and $p\in [1,\infty ]$ and  includes  the limiting cases $p=1,\infty$. 
That formula can be regarded as an integral version of another classical extension formula found by M.R. Hestenes\cite{hestenes}  in 1941  based on a linear combination of  a finite number of suitable reflections and which can be used  for the simpler case of domains of class $C^l$.  Loosely speaking,  Stein's formula involves an infinite number of reflections and this fact gives to Stein's Extension Operator a  global nature in the sense that the value of $Tf$ at a point $x\in\rr^n\setminus  \Omega$ depends on all values of $f$ along a line in $\Omega$, see \eqref{defT2}.  
Another extension operator was proposed by V.I. Burenkov\cite {burpaper2} in 1975, see also  \cite{burenkov2, burenkov}. Burenkov's Extension Operator is not universal since it depends on  $l\in \nn$. 
However, it has local nature in the sense that the values of $Tf$ around any point  $x\in\rr^n\setminus  \Omega$ depend on the values of $f$ around a finite number of reflected points.  This gives Burenkov's Extension Operator some flexibility and it  allows to treat   the case of domains of class $C^{0,\gamma}$ with $0<\gamma \le 1$ and domains with merely continuous boundaries  (with deterioration of smoothness of the extended functions), and anisotropic Sobolev spaces as well.  Such a  local feature was recently exploited in \cite{fanciullolamberti}  to prove that 
Burenkov's Extension Operator  preserves  Sobolev-Morrey spaces. More precisely,  
given $p\in [1,\infty [$,  a function $\phi$  from $\rr^+$ to $\rr^+$ and $\delta \in ]0, \infty ]$ one defines  the (generalized)  Morrey norm  of a function  $f \in L^p_{loc}(\Omega)$  by 
\begin{equation}
\label{morreynorm}
\| f\|_{M^{\phi,\delta}_p(\Omega)} := \sup_{x \in \Omega,\, 0<r<\delta} \left(  \frac{1}{\phi(r)}\int_{B_r(x)\cap \Omega} |f(y)|^p dy \right )^{\frac{1}{p}},
\end{equation}
and  simply  writes  $\| f\|_{M^{\phi}_p(\Omega)} $ if $\delta =\infty$.   The Morrey space $M^{\phi , \delta }_p(\Omega ) $ is the space of functions 
$f\in L^p_{loc}(\Omega )$ such that $\| f \|_{M^{\phi , \delta }_p(\Omega ) }<\infty $.  	Note that  if $\phi(r)=r^{\gamma}$ with $\gamma \geq 0$, then $M^{\phi }_p(\Omega ) $ are  the classical Morrey spaces introduced by C.B. Morrey~\cite{morrey} in 1938 and  also denoted by  $M^{\gamma}_p(\Omega) $  (obviously,  if $\gamma=0$ then $M_p^\phi(\Omega)=L^p(\Omega)$,  if $\gamma=n$
then 
$M_p^\phi(\Omega)=L^\infty(\Omega)$ and if  $\gamma>n$ then $M_p^\phi(\Omega)$ contains only the zero function).

It is proved in \cite{fanciullolamberti} that Burenkov's Extension Operator satisfies the following estimate
\begin{equation}
\| D^\alpha Tf\|_{M_p^{\phi,\delta}(\rr^n)}\le C\sum_{|\beta|\le |\alpha|}\|D^\beta f \|_{M_p^{\phi,\delta}(\Omega)}\, , \label{Sbound}
\end{equation}
for all  $f\in W^{l,p}(\Omega)$ and $|\alpha|\le l$, where $C>0$ is independent of $f$.   Moreover,  it is also proved  that if $\Omega $ is a bounded or an elementary/special unbounded domain,  then $C$ can be chosen to be independent of $\delta$ in which case estimate  $(\ref{Sbound})$ holds also if $\delta =\infty $. 
In particular, if $f\in W^{l,p}(\Omega )$ is such that $D^{\alpha }f\in M^{\phi , \delta }_p(\Omega )$ for all $|\alpha |\le l$ then $D^{\alpha }Tf\in  M^{\phi , \delta }_p(\rr^n)$ for all $|\alpha |\le l$.

Given the importance of Stein's Extension Operator and its wide use in mathematical analysis and  applications, it is clearly of interest to explore its fine properties as it has been done for Burenkov's Extension Operator.

In the present paper, we prove that also  Stein's operator satisfies estimate \eqref{Sbound}, hence it  preserves Sobolev-Morrey spaces.  We note that  although one usually expects that  an operator 
defined by a nice formula enjoys nice properties, the proof of our main result is not straightforward, the main obstruction being represented by the fact that, as we have said,  Stein's operator has a  global nature while  Morrey norms  have   a somewhat local genesis.

Needless to remark  the importance of  Morrey spaces. For example, they have been extensively used in  the study of  the local behaviour of solutions to elliptic and parabolic differential equations, see e.g., the survey papers \cite{lemarie, ragusa}. Moreover, they are the object of current research and many results have been recently obtained  in connection with the theory of singular integral operators,  and interpolation theory as well, see e.g., 
\cite{burenkov3, interpolation}.

With reference to the problem of  the extension  of  functions in Sobolev-Morrey spaces, we quote the paper \cite{vitolo} which is concerned with the case $l=1$: in that case Stein's operator is not required since the extension operator is provided by one reflection.  Moreover, we refer to \cite{koskela} for a description of extension domains for   certain Sobolev-Morrey spaces in the case $l=1$, $1\le p<n$,  $\phi(r)=r^{n-p}$. Finally, we refer to 
\cite{feffer, shva16, shva17} and the references therein for recent advances in the theory of extension operators. 

The main result of the present paper  is Theorem~\ref{Tlemma} which concerns special Lipschitz domains defined as epigraphs of Lipschitz continuous functions. Theorem~\ref{mains} is devoted to the general case. à

\section{Preliminaries}
\label{prelsec}

In this section we state a few results that will be used in the sequel. 
In particular, for the proof of Theorem~\ref{Tlemma}, we need the  Hardy-type inequality \eqref{hardy2}. Although there is a vast literature concerning Hardy and Hardy-type inequalities,  we include a proof for the specific case that we need  for the convenience of the reader.  We note that  setting  $a=c=0$ and $b=d=\infty$  in \eqref{hardy2}  gives the classical Hardy's Inequality
\[ \left ( \int_0^\infty \left(  \int_{x}^{\infty}x^\beta f(y)dy\right)^pdx \right)^{\frac{1}{p}}\le 
\frac{p}{\beta p +1} \left (\int_{0}^{\infty} (f(x)x^{\beta+1})^pdx\right)^{\frac{1}{p}},\]
for $\beta>-1/p.$

\begin{lemma}[Hardy-type inequality]\label{hardy}
	Let $\beta \in \rr$, $a,b,c,d \in \rr^+$ with $a<b$ and $c<d$ and let $p \in [1,\infty)$. Moreover, let $f$ be a non-negative measurable function in $(0,\infty)$. Then the following inequality holds
	\begin{equation}
	\left ( \int_a^b \left(  \int_{x+c}^{x+d}x^\beta f(y)dy\right)^pdx \right)^{\frac{1}{p}}\le 
	C \left (\int_{a+c}^{b+d} (f(x)x^{\beta+1})^pdx\right)^{\frac{1}{p}}, \label{hardy2}
	\end{equation}
	where
	$C=\int_{1+\frac{c}{b}}^{1+\frac{d}{a}}t^{-(\beta+1+1/p)}dt.$
\end{lemma}

\begin{proof}
	Applying the change of variable $y=tx$ in the inner integral of the left hand side of  \eqref{hardy2}   we get 
	\[
	\left ( \int_a^b \left(  \int_{x+c}^{x+d}x^\beta f(y)dy\right)^pdx \right)^{\frac{1}{p}} =
	\left ( \int_a^b \left(  \int_{1+c/x}^{1+d/x}x^{\beta+1} f(tx)dt\right)^pdx \right)^{\frac{1}{p}}
	\]
	that can be rewritten as
	\[ \left ( \int_a^b \left(  \int_{1+c/b}^{1+d/a}x^{\beta+1} \chi_A (t,x)f(tx)dt\right)^pdx \right)^{\frac{1}{p}},\]
	where  $A=\{ (t,x)\in {\mathbb{R}}^2 \ | \ 1+c/x\le t \le 1+d/x,\  x \in [a,b]\}.$  As customary, $\chi_C$ denotes 
	the characteristic function of a set $C$.
	Applying Minkowski's Integral Inequality yields
	\begin{eqnarray}\lefteqn{
		\left ( \int_a^b \left(  \int_{1+c/b}^{1+d/a}x^{\beta+1}  \chi_A (t,x)f(tx)dt\right)^pdx \right)^{\frac{1}{p}}  }\nonumber \\
	& & \qquad\qquad\qquad
	\le 
	\int_{1+c/b}^{1+d/a} \left ( \int_a^b \left( x^{\beta+1}  \chi_A (t,x)f(tx) \right)^p dx \right)^{\frac{1}{p}} dt. 
	\end{eqnarray}
	Let  $B=\{ (t,x) \in \rr^2 \ | \ a+c\le tx \le b+d\} $. 
	By observing that
	$A\subset B$, hence  $ \chi_A\le  \chi_B$, and by  applying the change of variables $u=tx$, we get
	\begin{eqnarray}\lefteqn{
		\int_{1+c/b}^{1+d/a} \left ( \int_a^b \left( x^{\beta+1} \chi_A(t,x)f(tx) \right)^p dx \right)^{\frac{1}{p}} dt} \nonumber \\
	& & \qquad
	\le
	\int_{1+c/b}^{1+d/a} \left ( \int_a^b \left( x^{\beta+1}\chi_B(t,x)f(tx) \right)^p dx \right)^{\frac{1}{p}}dt  \label{XY}  \nonumber \\
	& & \qquad =\int_{1+c/b}^{1+d/a} t^{-(\beta+1+1/p)}dt\left ( \int_{a+c}^{b+d} \left( u^{\beta+1} f(u) \right)^p du \right)^{\frac{1}{p}},
	\end{eqnarray}
	that is what we wanted to prove.
\end{proof}

Moreover, we shall use the following two lemmas the proofs of which are easy and are omitted.
Here and in the sequel $\nn_0$ denotes the set $\nn \cup \{ 0\}$. Furthermore, the elements of $\rr^n$ are denoted by $x=(\bar x, y)$ with $\bar x \in \rr^{n-1}$, $y \in \rr$, and it is always assumed $n\ge2.$ 

\begin{lemma}\label{derivatives}
	Let $f,h \in C^\infty(\rr^n)$, $\lambda \in \rr\setminus\{0\}$. Let $g\in C^\infty(\rr^n)$ be defined by $g(x)=f(\bar x, y+\lambda h(x))$ for all $x=(\bar x,y) \in \rr^{n}$. Then, for every $\alpha \in \nn_0^n$ and $x \in \rr^n$, $D^\alpha g(x)$ is a finite sum of terms of the following form
	\[
	c\lambda^s D^{\beta} f(\bar x, y+\lambda h(x))(D^{\gamma_1}h(x))^{n_1}\cdots (D^{\gamma_k}h(x))^{n_k}
	\]
	for some constant $c$, with $\beta,\gamma_i \in \nn_0^n $, $k,s,n_i \in \nn_0$ and $\beta,\gamma_i\neq0$, $k,s\ge 0$, $n_i>0$. It is meant that for $k=0$ no term $(D^{\gamma_i}h(x))^{n_i}$ is present. Moreover every term satisfies the following conditions
	\begin{enumerate}[a)]
		\item $
		\sum_{i=1}^k
		 n_i(|\gamma_i|-1)=|\alpha|-|\beta|$,
		\item  $s=0$ if and only if $k=0$.
	\end{enumerate}
\end{lemma}

\begin{lemma}\label{covering}
	Let $\Omega$ be a set in $\rr^n$ with diameter $D>0$ and let $k \in \nn$. Then there exists  $C_{n,k}\in {\mathbb{N}}$ depending only on $k$ and $n$ such that $\Omega$ can be covered by a collection of open balls $B_1,...,B_h$ centered in $\Omega$ with radius $D/k$ and $h \le C_{k,n}.$
\end{lemma}

\section{Stein's operator  on special Lipschitz domains}

In this section we consider  the case of  special Lipschitz  domains   $\Omega$  in $\rr^n$ of the form
\begin{equation} \label{omegaele}
\Omega=\{(\bar x, y) \in \rr^n \ | \ \psi(\bar x)<y \},
\end{equation}
where $\psi : \rr^{n-1} \to \rr$ is a Lipschitz continuous
function.
The Lipschitz constant of $\psi $ will be denoted by $M$ and will be called Lipschitz
bound of $\Omega$. 
Recall that the elements of  $\rr^n$ are denoted by  $x=(\bar x,y)$ with $\bar x\in\rr^{n-1}$ and $y\in \rr$ and that it is always assumed that $n\geq 2$. 

By $\Delta $  we denote a  fixed regularized distance from $\bar{\Omega}$. Namely, $\Delta \in C^{\infty}( \rr^n\setminus \bar{\Omega})$ and satisfies the following properties:
\begin{equation}
\label{rega}
c_1 d(x,\bar{\Omega} )\le \Delta(x)\le c_2d (x,\bar{\Omega} )
\end{equation}
and
\begin{equation}
\label{regb}
\left | D^\alpha \Delta(x) \right | \le B_\alpha d(x,\bar{\Omega} )^{1-|\alpha|},\  {\rm for\ all}\ \alpha \in \nn^n,
\end{equation}
for all $x \in  \rr^ n\setminus \bar{\Omega}$, where $B_\alpha$, $c_1$,$c_2$ are positive constants independent of $x$ and $\Omega $. Here 
$d(x,\bar{\Omega} )$ denotes  the Euclidean distance of  $x \in \rr^ n$ from  $\bar{\Omega} $.  
Moreover,  one can  prove that  
there exists a positive constant $  c_3 $, which depends only on  $M$  such that if $(\bar x, y) \in \rr^ n  \setminus \bar{\Omega} $ then 
\begin{equation} \label{deltabound}
c_3\Delta(\bar x,y)\ge \psi(\bar x)- y.
\end{equation}

We  denote by $\tau $ a fixed continuous real-valued function defined in $[1,\infty)$ satisfying the following properties
\begin{enumerate}[i)]
	\item $\tau(\lambda)=O(\lambda^{-N})$, as $\lambda \rightarrow \infty$ for every $N>0$,
	\item $\int_1^\infty \tau(\lambda)d\lambda=1$, $\int_1^\infty \lambda^k\tau(\lambda)d\lambda=0$, for every $k\in \nn$, $k\geq 1$.
\end{enumerate}

The existence of functions  $\Delta$ and $\tau$ is well-known, see e.g.,  \cite{stein}.

Recall that  if $\Omega$ is an open subset of $\rr^n$,  $W^{l,p}(\Omega)$  denotes the Sobolev space of functions 
 $f\in L^p(\Omega)$ with weak derivatives $D^{\alpha}f\in L^p(\Omega)$  for all  $|\alpha |\le l $, endowed with the norm $\| f\|_{W^{l,p}(\Omega)}=\sum_{0\le |\alpha |\le l}\|D^{\alpha}f\|_{L^p(\Omega)}$.

For an open subset $\Omega$ of $\rr^n$ we will also denote by $C^\infty_b(\bar \Omega)$ the set of functions $f\in C^\infty(\bar \Omega)$ such that $D^\alpha f$ is bounded for all $\alpha \in \nn^n_0$.

We are ready to state the following important  result by Stein.

\begin{theorem}[Stein's Extension Theorem - special case]\label{defT} Let $\Omega$, $\Delta$, $\tau$, $M$ and $c_3$ be as above.  For every function $f \in  C^{\infty }_b(\bar{\Omega})$ , define
	\begin{equation}
	Tf(\bar x, y)= \begin{cases}
	f(\bar x, y), & \text{ if } y\ge\psi(\bar x), \\
	\int_1^\infty f(\bar x, y+ \lambda \delta^*(\bar x,y))\tau(\lambda)d\lambda, & \text{ if } y<\psi(\bar x),		
	\end{cases}
	\label{defT2}\end{equation}
	where $\delta^*(\bar x,y)=2c_3 \Delta(\bar x, y).$ Then $Tf \in C^\infty(\rr^n)$ and  for every $l\in \mathbb{N}$,  $1\le p\le\infty$ we have
	\begin{equation}\label{est}
	\| Tf\|_{W^{l,p}(\rr^n)}\le  S\| f\|_{W^{l,p}(\Omega)} ,
	\end{equation}
	where $ S $ is a constant depending only on $n,l$ and $M.$ Moreover, for every $l\in \mathbb{N}$,  $1\le p\le\infty$,   $T$ admits a unique linear continuous extension from $C^{\infty }_b(\bar{\Omega})\cap W^{l,p}(\Omega)$ to the whole of 
	$W^{l,p}(\Omega)$, taking values in  $W^{l,p}(\rr^n)$ and satisfying estimate \eqref{est}.
\end{theorem}

\begin{remark}\label{remsob} A detailed proof of Theorem~\ref{defT} can be found in \cite{stein}. Since we shall need it later, here we briefly recall the procedure which allows to extend the operator $T$ defined by \eqref{defT2} from $ C^{\infty }_b(\bar{\Omega})\cap W^{l,p}(\Omega)$  to the whole of $W^{l,p}(\Omega)$  when $1\le p<\infty$. 
	We denote by $\Gamma$ the cone with vertex at the origin given by $\Gamma=\{(\bar x, y) \in \rr^n \ | \ M |\bar x|<|y|, y<0 \}$. Suppose now that $\eta \in C^\infty_c(\rr^n)$ is a non-negative function such that $\int_{\rr^n}\eta(x)dx=1$ and its support is contained in $\Gamma.$ For every $f \in W^{l,p}(\Omega)$ and every $\varepsilon>0$ we define
	\[f_\varepsilon(x)=\frac{1}{\varepsilon^n}\int_{\rr^n} f(x-y) \eta(y/\varepsilon)dy =\int_{\rr^n} f(x-\varepsilon y) \eta(y)dy.\]
	Notice that, since the support of $\eta$ is strictly inside $\Gamma$, the above integral is well defined for every $x$ in some neighbourhood of $\bar \Omega$ depending on $\varepsilon$. Hence $f_\varepsilon \in C^\infty_b(\bar\Omega)\cap W^{l,p}(\Omega)$, thus $Tf_\varepsilon$ is well defined. The Stein operator is then taken to be the limit in $W^{l,p}(\rr^n)$ of $Tf_\varepsilon$ as $\varepsilon \to 0.$ 
\end{remark}

In the proof of Theorem~\ref{Tlemma},  it will be convenient to consider a Morrey-type norm defined by means of cubes rather than balls. 
Namely,  given  $1\le p< \infty$,  a function  $\phi$ from $\rr^+$ to $\rr^+$, $\delta>0$ and a domain $\Omega$  in $\mathbb{R}^n$, we set 
\[ \|f\|_{M_{p,Q}^{\phi,\delta}(\Omega)}:=\sup_{x \in \Omega,0<r<\delta} \left(  \frac{1}{\phi(r)}\int_{Q_{2r}(x)\cap \Omega} |f(y)|^p dy \right )^{\frac{1}{p}},\]
for all  $f \in L^p_{loc}(\Omega)$  where $Q_{2r}(x) =\Pi_{k=1}^n]x_k-r,x_k+r [ $ is the open cube centered in $x$ of 
edge length $2r$.
It is easy  to see  that this  norm is equivalent to the norm defined by \eqref{morreynorm} and, in particular, that there exists a positive constant $ c_4 $ depending only on $n$ such that 
\begin{equation}\label{eqnorm} \| .\|_{M_{p}^{\phi,\delta}(\Omega)} \le \| .\|_{M_{p,Q}^{\phi,\delta}(\Omega)}\le c_4 \| .\|_{M_{p}^{\phi,\delta}(\Omega)}\, .
\end{equation}

We are now ready to state and prove the main result of this section.

\begin{theorem}\label{Tlemma}
	Let $1\le p<\infty$, $l\in \mathbb{N}$ and  $\phi$ a function from $\rr^+$ to $\rr^+$. Let  $\Omega$ be a special Lipschitz domain of $\rr^n$ with Lipschitz bound $M$. 
	Let  $T:W^{l,p}(\Omega)\to W^{l,p}(\rr^n)$ be the Stein's extension operator defined in Theorem~\ref{defT}. Then  there exists $C>0$  depending  only on $n,l$ and $M$ such that  
	\begin{equation}\label{Sbound2}
	\|D^\alpha Tf \|_{M^{\delta,\phi}_p(\rr^n)}\le C\sum_{|\beta|=|\alpha|} \|D^\beta f\|_{M^{\delta,\phi}_p(\Omega)}
	\end{equation}
	holds for all $f \in W^{l,p}(\Omega)$, $\delta>0$, and  $\alpha \in \nn_0^n$ with $|\alpha|\le l$. 
\end{theorem}

\begin{proof}   Let $\Omega$ be as in \eqref{omegaele} where $\psi$ is a Lipschitz function  with Lipschitz constant equal to $M$.
	We divide the proof in two steps.
	
	{\it Step 1. } 
	We prove inequality \eqref{Sbound2} for  functions $f \in  C^{\infty }_b(\bar{\Omega})\cap W^{l,p}(\Omega)$.
	
	First, we consider the  case $l=0$. By \eqref{eqnorm} it is enough to prove that for an arbitrary open cube $Q$ of edge length $r$ with $0<r<\delta$ and edges parallel to the coordinate axes we have
	\begin{equation}
	\phi(r/2)^{-1/p}\|Tf\|_{L^p(Q)} \le  C \| f\|_{M_{p,Q}^{\phi,\delta/2}(\Omega)} \label{cubestimate}
	\end{equation}
	for a constant $	C $ depending only on $n,M$. We remark that along the proof the value of the constant denoted by $C$ may vary, but it will remain dependent only on $l,n,M$.  Let  $\Omega^- = \{ (\bar x , y) \in \rr^n \ | \ \bar x \in \rr^{n-1}, \ y<\psi(\bar x) \}$.
	We find it convenient to discuss separately the following  three cases: 1. $ \overline Q \subset \Omega$ 2. $ \overline Q \subset \Omega^-$ 3. $ \overline Q\cap \{y=\psi(\bar x)\} \neq \emptyset.$
	
	Case 1. This case is trivial, since $Tf=f$ in $\Omega$ hence we have that 
	\[ \phi(r/2)^{-1}\int_Q |Tf(x)|^pdx =\phi(r/2)^{-1}\int_Q |f(x)|^pdx  \le  \| f\|_{M_{p,Q}^{\phi,\delta/2}(\Omega)}^p.\]

	Case 2. Let us write $Q$ as $Q=F\times (a-r,a) $ where $F$ is an open cube of $\rr^{n-1}$ of 
		edge length 
	$r$ and $a<\psi(\bar x)$ for every $\bar x \in F$. Fix now $(\bar x, y) \in Q$. By assumptions $\tau(\lambda)=O(\lambda^{-3})$, as $\lambda \rightarrow \infty.$ Hence using the definition of $Tf$ we have
	\begin{eqnarray} \lefteqn{|Tf(\bar x,y)| \le\int_1^\infty |f(\bar x, y+\lambda \delta^*(\bar x,y))||\tau(\lambda)|d\lambda  }\nonumber  \\
	& \qquad\qquad\qquad\qquad\qquad \le C \int_1^\infty |f(\bar x, y+\lambda \delta^*(\bar x,y))|\frac{1}{\lambda^3}d\lambda \label{bullet1}.
	\end{eqnarray}
	By applying the change of variable $s=y+\lambda \delta^*(\bar x,y)$, we get
	\begin{eqnarray}  |Tf(\bar x,y)|&\le &C \int_{y+\delta^*(\bar x, y)}^\infty |f(\bar x, s)|\frac{(\delta^*(\bar x, y))^2}{(s-y)^3}ds\nonumber \\
	&\le &  C \int_{2\psi(\bar x)-y}^\infty |f(\bar x, s)|\frac{(\psi(\bar x)-y)^2}{(s-y)^3}ds \label{bullet2}
	\end{eqnarray} 
	because $ 2c_2c_3 (\psi(\bar x)-y)\ge\delta^*(\bar x, y)\ge 2(\psi(\bar x)-y)$, which follows from \eqref{rega} and  \eqref{deltabound}. 
	By decomposing  the last integral in \eqref{bullet2} we obtain
	\[ |Tf(\bar x,y)|\le C \sum_{k=0}^\infty \int_{2\psi(\bar x)-y+kr}^{2\psi(\bar x)-y+(k+1)r} |f(\bar x, s)|\frac{(\psi(\bar x)-y)^2}{(s-y)^3}ds.\]
	Now by applying Minkowski's inequality for an infinite sum we get
	\begin{align*} & \|Tf(\bar x, y)\|_{L_y^p(a,a-r)} \\
	&\le C \sum_{k=0}^\infty \left( \int_{a-r}^{a}\left ( \int_{2\psi(\bar x)-y+kr}^{2\psi(\bar x)-y+(k+1)r} \frac{|f(\bar x, s)|(\psi(\bar x)-y)^2}{(s-y)^3}ds\right)^pdy \right)^{\frac{1}{p}}. \addtag \label{sum2} 
	\end{align*}
	We plan to estimate each summand in the right hand side of \eqref{sum2}. First of all, by applying  the change of variable $y=\psi(\bar x)-z$ we get that 
	each summand equals
	\begin{equation}
	\label{summand} 
	\left( \int_{\psi(\bar x)-a}^{\psi(\bar x)-a+r}\left (\int_{\psi(\bar x)+z+kr}^{\psi(\bar x)+z+(k+1)r} |f(\bar x, s)|\frac{z^2}{(s-\psi(\bar x)+z)^3}ds\right)^pdz \right)^{\frac{1}{p}}.
	\end{equation}
	Then we apply the change of variable $t=s-\psi(\bar x)$ to the inner integral of \eqref{summand}, obtaining
	\begin{align*} 
	& \left( \int_{\psi(\bar x)-a}^{\psi(\bar x)-a+r}\left (\int_{z+kr}^{z+(k+1)r} |f(\bar x, t+\psi(\bar x))|\frac{z^2}{(t+z)^3}dt\right)^pdz \right)^{\frac{1}{p}} \\ 
	&\le \left( \int_{\psi(\bar x)-a}^{\psi(\bar x)-a+r}\left (\int_{z+kr}^{z+(k+1)r} |f(\bar x, t+\psi(\bar x))|\frac{z^2}{t^3}dt\right)^pdz \right)^{\frac{1}{p}},
	\end{align*}
	where we have used that  $z\ge\psi(\bar x)-a>0.$  Next by 
	Lemma \ref{hardy}
	(with $f(t)$ replaced by $|f(\bar x, \psi(\bar x) +t)|/t^3$, $a$ replaced by $\psi(\bar x)-a$, $b$ replaced by $\psi (\bar x )-a+r $, $c=kr$, $d=(k+1)r$, $\beta=2$)  we have	
	\begin{align*}
	&\left( \int_{\psi(\bar x)-a}^{\psi(\bar x)-a+r}\left (\int_{z+kr}^{z+(k+1)r} |f(\bar x, t+\psi(\bar x))|\frac{z^2}{t^3}dt\right)^pdz \right)^{\frac{1}{p}}  \\
	&\le \int_{1+k\alpha/(\alpha+1)}^{1+(k+1)\alpha}\frac{1}{t^{3+1/p}}dt \left ( \int_{\psi(\bar x)-a+kr}^{\psi(\bar x)-a+(k+2)r}|f(\bar x, z+\psi(\bar x))|^p  dz \right)^{\frac{1}{p}}\\
	&\le \int_{1+k\alpha/(\alpha+1)}^{1+(k+1)\alpha}\frac{1}{t^{3}}dt \left ( \int_{\psi(\bar x)-a+kr}^{\psi(\bar x)-a+(k+2)r}|f(\bar x, z+\psi(\bar x))|^p  dz \right)^{\frac{1}{p}}\\
	&= s_k(\bar x)\left ( \int_{\psi(\bar x)-a+kr}^{\psi(\bar x)-a+(k+2)r}|f(\bar x, z+\psi(\bar x))|^p  dz \right)^{\frac{1}{p}}
	\end{align*}
	where $\alpha=\alpha(\bar x)=r/(\psi(\bar x)-a)$ and 
	\begin{equation}\label{skfom}
	s_k(\bar x)=\frac{\alpha(\alpha+2)}{2((k+1)\alpha+1)^2}\, .
	\end{equation}
	Using this estimate in \eqref{sum2} we get
	\begin{align*}
	\left(\int_{a-r}^{a}|Tf(\bar x,y)|^p dy\right)^{\frac{1}{p}} & \le C  \sum_{k=0}^\infty s_k(\bar x) \left ( \int_{\psi(\bar x) -a+kr}^{\psi(\bar x) -a+(k+2)r}|f(\bar x, z+\psi(\bar x))|^p  dz \right) ^{\frac{1}{p}} \\
	&= C \sum_{k=0}^\infty s_k(\bar x) \left ( \int_{2\psi(\bar x) -a+kr}^{2\psi(\bar x) -a+(k+2)r}|f(\bar x, y)|^p  dy \right) ^{\frac{1}{p}}. \addtag \label{sum3}
	\end{align*}
	We claim now that there exists a sequence $s_k(Q)$, not depending on $\bar x$ such that $s_k(\bar x) \le s_k(Q)$ for every $\bar x \in F$ and such that
	\begin{equation}
	\sum_{k=0}^\infty s_k(Q)\le \tilde C
	\end{equation}	
	where $\tilde C$ is a constant depending only on $n$ and $M$. To see this, observe that the function $\alpha=\alpha(\bar x) : \bar F \to \rr$ is continuous and strictly positive, hence it admits a minimum $\ell >0$ and a maximum $L$. We distinguish two cases: $\ell >1/(2\sqrt nM)$ and $\ell  \le 1/(2\sqrt nM)$.
	If $\ell >1/(2\sqrt nM)$ we get
	\[    s_k(\bar x)=  \frac{\alpha(\alpha+2)}{   2  ((k+1)\alpha+1)^2} \le  \frac{\alpha(\alpha+2)}{   2   (k+1)^2\alpha^2}= \frac{1+\frac{2}{\alpha}}{2  (k+1)^2} \le \frac{1+4\sqrt n M}{2  (k+1)^2}=:s_k(Q).\]
	That is what we wanted. We consider now the case $\ell \le 1/(2\sqrt nM)$. We first observe that since the Lipschitz constant of  $\psi$ is  $M$, we have that 
	\[ \frac{\psi(\bar x_1)-a}{r}-\frac{\psi(\bar x_2)-a}{r}\le \sqrt n M\]
	for every $\bar x_1, \bar x_2 \in \bar F,$ that implies
	\[ \frac{1}{\ell }-\frac{1}{L}\le \sqrt n M, \]
	ans thus
	\[ L \le \frac{\ell}{1-\ell \sqrt nM}  \le 2\ell .\]
	Now we can perform the following estimate
	\begin{eqnarray*}\lefteqn{
		s_k(\bar x)=   \frac{\alpha(\alpha+2)}{  2   ((k+1)\alpha+1)^2} \le   \frac{2\ell  (2\ell+2)}{((k+1)\ell +1)^2} } \\
		& & 
		\qquad\qquad\qquad
		  \le\left(\frac{1}{\sqrt nM}+2\right ) \frac{2\ell }{((k+1)\ell+1)^2}:=s_k(Q). 
	\end{eqnarray*}	
	Observe now that 
	\[\sum_{k=0}^\infty \frac{\ell }{((k+1)\ell +1)^2}=   \sum_{k=0}^\infty \int_{\ell k}^{\ell  (k+1)} \frac{1}{((k+1)\ell  +1)^2}dt \le \int_0^\infty \frac{1}{(t+1)^2}dt= 1.\]
	This proves our claim.
	Applying the estimate  $s_k(\bar x) \le s_k(Q)$ in  \eqref{sum3} we get
	\begin{align*}
	\left(\int_{a-r}^{a}|Tf(\bar x,y)|^p dy\right)^{\frac{1}{p}} & \le C \sum_{k=0}^\infty s_k(Q) \left ( \int_{2\psi(\bar x) -a+kr}^{2\psi(\bar x) -a+(k+2)r}|f(\bar x, y)|^p  dy \right) ^{\frac{1}{p}}. 
	\end{align*}	
	Taking the $L^p$ norm on $F$ on both sides  and applying again Minkowski inequality we obtain
	\begin{align*}
	\|Tf\|_{L^p(Q)}  &\le C \sum_{k=0}^\infty s_k(Q) \|f\|_{L^p(S_k)}. \addtag \label{finalsum}
	\end{align*}
	where $S_k=\{ (\bar x, y) \in \rr^n \ | \ \bar x \in F ,\  2\psi(\bar x) -a+kr < y < 2\psi(\bar x) -a+(k+2)r \}$. The set $S_k$ has the following two properties 
	\begin{equation}
	\label{propertiesS}
	{\rm diam}(S_k)\le   c_5 r,\ \ {\rm and}\ \ S_k\subset \Omega\, ,
	\end{equation}
	where  $ c_5$ is a constant depending only on $n$ and $M$.   Recall that ${\rm diam }(A)$ denotes the diameter of a set $A$.
	To prove the first property in \eqref{propertiesS}, we consider  two arbitrary points  $(\bar x_1,y_1),(\bar x_2,y_2)$ in $S_k$, we assume directly that  $y_2\ge y_1$ and we easily see that 
	\begin{align*}
	y_2-y_1 &\le2\psi(\bar x_2)-a+(k+2)r -(2\psi(\bar x_1) -a+kr) \\
	&=2 (\psi(\bar x_2)-\psi(\bar x_1))+2r\le 2M|\bar x_1 - \bar x_2|+2r\le 2r(M\sqrt{n-1}+1 ).
	\end{align*}
	To prove  the second property in \eqref{propertiesS},  just notice that for every $(\bar x, y ) \in S_k$ we have $y>2\psi(\bar x)-a>\psi(\bar x)$.   The first property in \eqref{propertiesS}
	together with Lemma~\ref{covering} implies that there exists a collection of open cubes $Q_{1,k},...,Q_{m,k}$ centred in $S_k$   and with edges of length 
	$r$ that covers $S_k$, with $m \in \nn$ depending only on $M$ and $n$. Hence
	$ S_k \subset \bigcup_{i=1}^m (Q_{i,k}\cap \Omega) $
	and  the second property in \eqref{propertiesS}  guarantees that every cube $Q_{i,k}$ is centered in $\Omega$. 
	Therefore by \eqref{finalsum}  we get
	\[ \| Tf\|_{L^p(Q)} \le C \sum_{k=0}^\infty s_k(Q) (\|f\|_{L^p(Q_{1,k}\cap \Omega)}+...+\|f\|_{L^p(Q_{m,k}\cap \Omega)}),\]
	hence dividing in both sides by $\phi(r/2)^{\frac{1}{p}}$ we obtain
	\[\phi(r/2)^{-1/p}\|Tf\|_{L^p(Q)} \le C \sum_{k=0}^\infty s_k(Q) \| f\|_{M^{\phi,\delta/2}_{p,Q}(\Omega)} \le  \| f\|_{M^{\phi,\delta/2}_{p,Q}(\Omega)}C  \, ,\]
	that is \eqref{cubestimate}.

	Case 3. Again,  we write $Q$ as $Q=F \times (a-r,a)$ as above and we set $Q^+=Q\cap\Omega$ and $Q^-=Q\cap\Omega^-$. 
	Moreover, $Q^-$ can be further decomposed as $Q^-=Q^-_1 \cup Q^-_2$ where $Q^-_1=\{ (\bar x,y) \in Q^- \ | \ \psi(\bar x)>a \}$ and $Q^-_2=\{ (\bar x,y) \in Q^- \ |\  a-r \le \psi(\bar x)\le a \}$. Note that 
	$\|Tf\|_{L^p(Q)}\le\|f\|_{L^p(Q^+)}+\|Tf\|_{L^p(Q^-)}$ and that it's immediate to verify that  $\| f\|_{L^p(Q^+)} \le  C \phi(r/2)^{\frac{1}{p}} \| f\|_{M_p^{\phi,\delta/2}(\Omega)}$, where $C$ depends only on $n$. 
	Hence it remains to estimate $\|Tf\|_{L^p(Q^-)}$. 
	Define the two Borel sets $\mathcal{S}_1:=\{\bar x \in \bar F \ | \ \psi(\bar x) > a\}$ and $\mathcal{S}_2:=\{\bar x \in \bar F \ | \ a-r\le \psi(\bar x) \le a\}$
	and  note that
	\begin{align*}  \|Tf\|^p_{L^p(Q^-)} &= \|Tf\|^p_{L^p(Q_1^-)} +\|Tf\|^p_{L^p(Q_2^-)} \\
	&=\int_{{\mathcal{S}}_1} \int_{a-r}^{a} |Tf(\bar x,y)|^pdy d\bar x+\int_{{\mathcal{S}}_2} \int_{a-r}^{\psi(\bar x)} |Tf(\bar x,y)|^pdy d\bar x
	\end{align*}
 	For every $\epsilon>0$ we define now the compact set ${\mathcal{S}}_1^\epsilon:=\{\bar x \in \bar F \ | \ \psi(\bar x) \ge a+\epsilon\}$ and we notice that
 	if $\bar x \in {\mathcal{S}}^\epsilon_1$ then \eqref{sum3} holds. Morover in the set $\mathcal{S}_1^\epsilon$, the function $\alpha(\bar x)=r/(\psi(\bar x)-a)$ is continuous and strictly positive and admits a minimum $\ell(\epsilon)>0$ and a maximum $L(\epsilon)$.  Thus by arguing as in Case 2 we can prove the existence of quantities $s_k(\epsilon,Q)$ such thath $s_k(\bar x)\le s_k(\epsilon,Q)$ and
 	\[ \sum_{k=0}^\infty s_k(\epsilon,Q)\le \tilde C \]
 		 where $\tilde C$ depends only on $n$ and $M$.
 	Hence taking the $L^p$ norm on $\mathcal{S}_1^\epsilon$ in \eqref{sum3} we obtain
	 	\begin{align*}
	 &\left(\int_{S_1^\varepsilon}\int_{a-r}^a|Tf(\overline x,y)|^p dy d\overline x\right)^{\frac{1}{p}} \le C  \sum_{k=0}^\infty s_k(\epsilon,Q) \|f\|_{L^p(S'_k)}\, , 
	 \end{align*}
	 where $S'_k=\{(\bar x, y)\in \rr^n:\ \bar x\in  {\mathcal{S}_1^\epsilon},\ \ \psi (\bar x)+a +kr<y< \psi (\bar x)+a +(k+2)r   \}$.
	 We observe that the sets $S'_k$ satisfy the same properties \eqref{propertiesS} of the sets $S_k$ considered in Case 2, hence dividing by $\phi(r/2)^{-1/p}$ we infer
	 \begin{align*}
	 & \phi(r/2)^{-1/p}\left(\int_{S_1^\varepsilon}\int_{a-r}^a|Tf(\overline x,y)|^p dy d\overline x\right)^{\frac{1}{p}} \le C \tilde C \| f\|_{M_{p,Q}^{\phi,\delta/2}(\Omega)}\, , 
	 \end{align*}
 	Recall that Theorem  \ref{defT} guarantees that $Tf \in L^p(\rr^n)$, hence by  Dominated Convergence Theorem we can let $\epsilon$ go to zero  to get
 	\begin{equation}
 	\phi(r/2)^{-1/p} \| Tf\|_{L^p(Q^-_1)} \le C \| f\|_{M_{p,Q}^{\phi,\delta/2}(\Omega)} \label{q1}\, ,
 	\end{equation}
	where $C$  depends only on $n$ and $M$. If instead $\bar x \in {\mathcal{S}}_2$, since $\psi(\bar x)\le a$, we have
	\begin{equation}
	\int_{a-r}^{\psi(\bar x)} |Tf(\bar x,y)|^pdy \le \int_{\psi(\bar x)-r}^{\psi(\bar x)} |Tf(\bar x,y)|^pdy. \label{psia}
	\end{equation}
	Now  for any $\epsilon >0$,  by \eqref{sum3} with $a$ replaced by $\psi(\bar x)-\epsilon$, we obtain
	\begin{align*}
	\left(\int_{\psi(\bar x)-\epsilon-r}^{\psi(\bar x)-\epsilon}|Tf(\bar x,y)|^p dy\right)^{\frac{1}{p}} \le C  \sum_{k=0}^\infty s_k( \epsilon) \left ( \int_{\psi(\bar x) +\epsilon +kr}^{\psi(\bar x) +\epsilon +(k+2)r}|f(\bar x, y)|^p  dy \right) ^{\frac{1}{p}},
	\end{align*}
	where $s_k(\epsilon)$ has the same expression as in \eqref{skfom}, with $\alpha =r/\epsilon$. We remark that, although the value of $\alpha $ blows up as $\epsilon$ goes to zero, the quantity $s_k(\epsilon)$ tends to $\frac{1}{2(k+1)^2}$ that has a finite sum. More precisely we have that, if $\alpha(\epsilon)>1$, then $s_k(\epsilon) \le \frac{3}{2(k+1)^2}$ and if $\alpha(\epsilon)\le 1$ then $s_k(\epsilon) \le \frac{3\alpha}{2((k+1)\alpha+1)^2}$. Moreover we have showed in Case 2 that $\sum_k \frac{\alpha}{((k+1)\alpha+1)^2}\le 1$ for any value of $\alpha>0$. In particular we deduce that for any $\epsilon>0$
	\[\sum_{k=0}^\infty s_k(\epsilon) \le \tilde C\]
	for some constant $\tilde C>0$ independent of $\epsilon$. 
	Taking now the $L^p$ norm over ${\mathcal{S}}_2$ on both sides of the previous integral inequality we obtain
	\begin{align*}
	&\left(\int_{S_2}\int_{\psi(\overline x)-\epsilon-r}^{\psi(\overline x)-\epsilon}|Tf(\overline x,y)|^p dy d\overline x\right)^{\frac{1}{p}} \le  C \sum_{k=0}^\infty s_k(\epsilon) \|f\|_{L^p(S''_k)}\, , 
	\end{align*}
	where $S''_k=\{(\bar x, y)\in \rr^n:\ \bar x\in  {\mathcal{S}}_2,\ \ \psi (\bar x)+\epsilon +kr<y< \psi (\bar x)+\epsilon +(k+2)r   \}$.
	We observe that the sets $S''_k$ satisfy the same properties \eqref{propertiesS} of the sets $S_k$ considered in Case 2, therefore
	\begin{equation}\label{prelimit} \left(\frac{1}{\phi(r/2)}\int_{S_2} \int_{\psi(\bar x)-\epsilon-r}^{\psi(\bar x)-\epsilon}|Tf(\bar x,y)|^p dy d\bar x\right)^{\frac{1}{p}}\le C \tilde C  \| f\|_{M_{p,Q}^{\phi,\delta/2}(\Omega)} 
	\end{equation}
	with $C,\tilde C$  depending only on $n$ and $M$. Again, since $Tf \in L^p(\rr^n)$, by  Dominated Convergence Theorem we can let $\epsilon$ go to zero in  \eqref{prelimit} obtaining
	\[
	\left(\frac{1}{\phi(r/2)}\int_{S_2} \int_{\psi(\bar x)-r}^{\psi(\bar x)}|Tf(\bar x,y)|^p dy d\bar x\right)^{\frac{1}{p}}\le C \| f\|_{M_{p,Q}^{\phi,\delta/2}(\Omega)} . 
	\]
	Combining the above inequality with \eqref{psia} we obtain
	\begin{equation}  \left(\frac{1}{\phi(r/2)}\int_{S_2} \int_{a-r}^{\psi(\bar x)}|Tf(\bar x,y)|^p dy d\bar x\right)^{\frac{1}{p}}\le C \| f\|_{M_{p,Q}^{\phi,\delta/2}(\Omega)} . \label{q2}
	\end{equation}
	Thus  putting together  \eqref{q1} and \eqref{q2} gives
	\begin{align*}
	\phi(r/2)^{-1/p} \| Tf\|_{L^p(Q^-)} &\le C \| f\|_{M_{p,Q}^{\phi,\delta/2}(\Omega)} 
	\end{align*}
	and   this concludes the proof in  Case 3.
	
	We consider now the case $l>0.$ By  \eqref{eqnorm} it's again enough to prove that for an arbitrary open cube $Q$ of 
	edge length 
	$r$ contained in $\rr^n$ we have the estimate
	$
	\phi(r/2)^{-1/p}\|D^\alpha Tf\|_{L^p(Q)}  \le C  \sum_{|\beta| =|\alpha|}\| D^\beta f\|_{M_{p,Q}^{\phi,\delta/2}(\Omega)}
	$
	for a constant $C $ depending only on $l,n,M$. We will consider the same three cases that appeared with $l=0$. Since $D^\alpha Tf=D^\alpha f$ in $\Omega$, the first case is trivial as before. We will see that the Cases 2 and 3 also follow from the computations done with $l=0$. We start by  observing that by the boundedness of $f$ and all its derivatives we can differentiate under the integral sign to get 
	$D^\alpha Tf(\bar x,y)= \int_1^\infty D^\alpha g_\lambda(\bar x,y) \tau(\lambda) d\lambda $ 
	for every $(\bar x, y) \in \Omega^-$, where $g_\lambda(\bar x,y)=f(\bar x, y+\lambda \delta^*(\bar x, y))$. By Lemma \ref{derivatives} $D^\alpha g_\lambda(\bar x,y)$ is a finite sum of terms of the type 
	\[ \widetilde c\lambda ^s D^\beta f(\bar x, y+\lambda \delta^*(\bar x, y)(D^{\gamma_1}\delta^*(x))^{n_1}\cdots (D^{\gamma_k}\delta^*(x))^{n_k}.\]
	For each of these terms we also set
	\begin{align*}
	&T_{s,\beta,(\gamma_1,n_1),...,(\gamma_k,n_k)}(x) \\
	&=  (D^{\gamma_1}\delta^*(x))^{n_1}\cdots (D^{\gamma_k}\delta^*(x))^{n_k}  \int_1^\infty \lambda ^s D^\beta f(\bar x, y+\lambda \delta^*(\bar x, y))\tau(\lambda) d\lambda.
	\end{align*}
	Thus $D^\alpha Tf(\bar x,y)$ is a finite sum of terms of the type $\widetilde c T_{s,\beta,(\gamma_1,n_1),...,(\gamma_k,n_k)}(x)$. Now, since the constants $\widetilde c$ and the number of terms of the sum depends only on $l$ and $n$, we just need to show that
	\begin{equation}
	\phi(r/2)^{-1/p}\|T_{s,\beta,(\gamma_1,n_1),...,(\gamma_k,n_k)}\|_{L^p(Q)} \le  C  \sum_{|\gamma| = |\alpha|}\| D^\gamma f\|_{M_{p,Q}^{\phi,\delta/2}(\Omega)} \label{Qderbound}
	\end{equation}
	for a constant $C $ depending only on $l,n,M$.
	
	We start by assuming that the multi-index $\beta $ on the left hand side of \eqref{Qderbound} satisfies $|\beta|=|\alpha|.$ By the property a) in Lemma~\ref{derivatives} and by the estimates of the derivatives of $\delta^*(=2c_3\Delta)$ given by \eqref{regb} we have that 
	\begin{align*}
	|T_{s,\beta,(\gamma_1,n_1),...,(\gamma_k,n_k)}(x) | &\le C  \int_1^\infty\lambda^s |D^\beta f(\bar x, y+\lambda \delta^*(\bar x, y))| |\tau(\lambda)|d\lambda \\
	&\le C  \int_1^\infty|D^\beta f(\bar x, y+\lambda \delta^*(\bar x, y))| \frac{1}{\lambda^3}d\lambda
	\end{align*}
	where  $C $ depends only  on $n$ and $M.$ We are now in the same situation as in the second inequality of \eqref{bullet1} (with $f$ replaced by $D^\beta f$). Hence we can proceed to prove the estimate in the same way as in case $l=0$ to get 
	$\phi(r/2)^{-1/p}\|T_{s,\beta,(\gamma_1,n_1),...,(\gamma_k,n_k)}\|_{L^p(Q)} \le C \| D^\beta f \|_{M_{p,Q}^{\phi,\delta/2}(\Omega)} $
	for every $Q$ in Case 2 and Case 3, where $C $ depends only on $n$ and $M.$ This proves \eqref{Qderbound} when $|\beta|=|\alpha|.$
	
	Suppose now that $|\beta|<|\alpha|$. We recall that, by Lemma~\ref{derivatives}, $|\beta|<|\alpha|$ implies that $s,k>0$. Arguing as above, using again \eqref{regb} and  Lemma~\ref{derivatives} we get
	\begin{align*}
	&|T_{s,\beta,(\gamma_1,n_1),...,(\gamma_k,n_k)}(x) | \\
	&\le  \frac{C}{d(x,\bar \Omega)^{|\alpha|-|\beta|}} \left|\int_1^\infty \lambda^sD^\beta f(\bar x, y+\lambda \delta^*(\bar x, y))\tau(\lambda)d\lambda \right| \\
	& \le  \frac{C}{(\psi(\bar x)-y)^{|\alpha|-|\beta|}} \left |\int_1^\infty \lambda^s D^\beta f(\bar x, y+\lambda \delta^*(\bar x, y)) \tau(\lambda)d\lambda \right |\addtag \label{zzz} .   
	\end{align*}
	Where $C $ depends only on $n,l$ and $M$. By applying Taylor's formula  about the point  $t=\delta^*=\delta^*(\bar x,y)$  up to order $m=|\alpha|-|\beta|$ and  with  remainder in integral form for the function $t \mapsto D^\beta f(\bar x, y+t)$,  we get
	\begin{align*}
	D^\beta f(\bar x, y+\lambda \delta^*) = &\sum_{j=0}^{m-1} \frac{(\lambda \delta^*-\delta^*)^j}{j!}\frac{\partial^j D^\beta f}{\partial x_n^j}(\bar x,y+\delta^*) \\
	&\ \ +\int_{\delta^*}^{\lambda \delta^*} \frac{(\lambda \delta^*-t)^{m-1}}{m!}\frac{\partial^{m} D^\beta f}{\partial x_n^{m} }(\bar x,y+t)dt. 
	\end{align*}
	We observe that the terms inside the  sum in the right hand side do not  give any contribution in \eqref{zzz}, since
	\begin{align*} &\int_1^\infty \frac{\lambda^s(\lambda \delta^*-\delta^*)^j}{j!}\frac{\partial^j D^\beta f}{\partial x_n^j}(\bar x,y+\delta^*)\tau(\lambda)d\lambda \\
	&=\frac{\partial^j D^\beta f}{\partial x_n^j}(\bar x,y+\delta^*) \frac{(\delta^*)^j}{j!} \int_1^\infty \lambda^s(\lambda-1)^j\tau(\lambda)d\lambda=0
	\end{align*}
	by the property $ii)$ of $\tau$  and the fact that $s>0$. Hence combining this with \eqref{zzz} we obtain
	\begin{align*}
	&|T_{s,\beta,(\gamma_1,n_1),...,(\gamma_k,n_k)}(x) | \\
	&\le \frac{C }{(\psi(\bar x)-y)^{m}} \left |\int_1^\infty \int_{\delta^*}^{\lambda \delta^*} \frac{(\lambda \delta^*-t)^{m-1}}{m!}\frac{\partial^{m} D^\beta f}{\partial x_n^{m} }(\bar x,y+t)dt \lambda^s \tau(\lambda)d\lambda\right |.
	\end{align*}
	Observing that $(\lambda\delta^*-t)^{m-1}\le (\lambda\delta^*)^{m-1}$, recalling that $  2c_2c_3 (\psi(\bar x)-y)\ge \delta^*$ and using the change of variable $u=y+t$ we get
	\begin{align*}
	|T_{s,\beta,(\gamma_1,n_1),...,(\gamma_k,n_k)}(x) | \le \frac{C }{\delta^*}\int_1^\infty \int_{y+\delta^*}^{y+\lambda \delta^*} \left|\frac{\partial^{m} D^\beta f}{\partial x_n^{m} }(\bar x,u)\right |\lambda^{s+m-1} |\tau(\lambda)|du d\lambda.
	\end{align*}
	Performing a change of order of integration we deduce    
	\[ |T_{s,\beta,(\gamma_1,n_1),...,(\gamma_k,n_k)}(x) | \le\frac{C }{\delta^*}\int_{y+\delta^*}^{\infty} \left |\frac{\partial^{m} D^\beta f}{\partial x_n^{m} }(\bar x,u)\right | \int_{(u-y)/\delta^*}^\infty |\lambda^{s+m-1} \tau(\lambda)| d\lambda du.\]   
	Finally recalling
	that 

	$\tau(\lambda)=O(\lambda^{-m-s-1})$ as $\lambda \rightarrow \infty$,  we can write
	\[ |T_{s,\beta,(\gamma_1,n_1),...,(\gamma_k,n_k)}(x) |  \le C \int_{y+\delta^*}^{\infty} \left |\frac{\partial^{m} D^\beta f}{\partial x_n^{m} }(\bar x,u)\right | \frac{(\delta^*)^2}{(u-y)^3} du.\]  
	We observe that we are now in the same situation as in the first inequality of \eqref{bullet2} of the case $l=0$ (with $f$ replaced by $\frac{\partial^{m} D^\beta f}{\partial x_n^{m} }$) and the same computations lead us to the inequality
	
	\[\phi(r/2)^{-1/p}\|T_{s,\beta,(\gamma_1,n_1),...,(\gamma_k,n_k)}\|_{L^p(Q)}  \le  C \bigl\| \frac{\partial^{m} D^\beta f}{\partial x_n^{m} } \bigr\|_{M_{p,Q}^{\phi,\delta/2}(\Omega)} \] 
	for every $Q$ in Case 2 and Case 3, where $ C$ depends only on $n,l$ and $M$. This concludes the proof of \eqref{Qderbound} and of the case $l>0$ since $m+|\beta|=|\alpha|$. {\it Step 1} is now complete.    
	
	{\it Step 2.} We prove inequality \eqref{Sbound2} for functions $f \in W^{l,p}(\Omega)$. 
	Recall  the definition of the operator $S$ explained in Remark~\ref{remsob}.  Let  $\Gamma$ to be the cone $\Gamma=\{(\bar x, y) \in \rr^n \ | \ M |\bar x|<|y|, y<0 \}$ and let $\eta \in C^\infty_c(\rr^n)$ be a function with  $\int_{\rr^n} \eta(x)dx =1$      and support contained in $\Gamma.$ Then, given $f \in W^{l,p}(\Omega)$, $Sf$ is defined to be the limit in $W^{l,p}(\rr^n)$ of $Tf_\varepsilon$ as $\varepsilon \to 0,$ where $f_\varepsilon(x)=1/\varepsilon^n \int_{\rr^n} f(x-y)\eta(y/\varepsilon)$ for every $x$ in an appropriate neighbourhood of $	\bar \Omega$. We claim that for every $f \in W^{l,p}(\Omega)$, $\delta>0$ and $|\alpha|\le l$
	\begin{equation}
	\|  D^\alpha f_\varepsilon\|_{M^{\phi,\delta}_p(\Omega)} \le \| D^\alpha f\|_{M^{\phi,\delta}_p(\Omega)} \label{epsbound}.
	\end{equation}
	To see this first we notice that $D^\alpha f_\varepsilon(x)=1/\varepsilon^n \int_{\rr^n} D^\alpha f(x-y)\eta(y/\varepsilon)dy$ for every $x \in \Omega.$ Let now $B_{x_0}(r)$ a ball centered in $\Omega$ of radius $0<r<\delta$ and set $\eta_\varepsilon(x)=\varepsilon^{-n}\eta(x/\varepsilon)$. By Minkowski's integral inequality   
	\begin{align*}\lefteqn{
		\|D^\alpha f \ast \eta_\varepsilon \|_{L^p(B_r(
			x_0			
			)\cap \Omega)}} \\ & & 
	\le \int_{\rr^n}\eta_\varepsilon(y) \|D^\alpha f  \|_{L^p(B_r(
		x_0		
		-y)\cap \Omega)}  dy \le\phi(r)^{1/p} \| D^\alpha f\|_{M_p^{\phi,\delta}(\Omega)}
	\end{align*}
	because $B_r(x_0)\cap \Omega -y \subset B_r(x_0-y)\cap \Omega$ and $x_0-y \in \Omega$ for every $x_0 \in \Omega$ and $y \in \Gamma.$ This proves \eqref{epsbound}. Now combining \eqref{epsbound} with \eqref{Sbound2} we get the inequality 
	$ \| D^\alpha Tf_\varepsilon\|_{M_p^{\phi,\delta}(\rr^n)} \le C\sum_{|\beta|=|\alpha|}\| D^\beta f \|_{M_p^{\phi,\delta}(\Omega)}, $  
	for every $\varepsilon>0$ and every $|\alpha|\le l$, with $ C$ independent of $\varepsilon$. In particular, for every ball $B$ in $\rr^n$ of radius $r \in ]0,\delta[$ we have
	\begin{equation}
	\phi(r)^{-1/p}\| D^\alpha Tf_\varepsilon\|_{L^p(B)}   \le C\sum_{|\beta|=|\alpha|}\|D^\beta f \|_{M_p^{\phi,\delta}(\Omega)}. \label{balleps}
	\end{equation}
	Since $ Tf_\varepsilon$ converges to $Tf$ in $W^{l,p}(\rr^n)$, then  $D^\alpha Tf_\varepsilon$ converges to $D^\alpha Tf$ in $L^p(\rr^n)$ for every $|\alpha|\le l$ and as a consequence also in $L^p(B)$ for every ball $B$. Hence we can pass to the limit as $\varepsilon \to 0$ in \eqref{balleps} and obtain the estimate 
	$\phi(r)^{-1/p}\| D^\alpha Tf\|_{L^p(B)}    \le C \sum_{|\beta|= |\alpha|}\|D^\beta f \|_{M_p^{\phi,\delta}(\Omega)} $
	for every ball $B$ of radius $r$  and with $C$ depending only on $l,n$ and $M$. This concludes the proof.

\end{proof}

\begin{remark}\label{rotlip}
	Let $\Omega$ be a domain in $\rr^n$ and suppose that there exists a special Lipschitz domain $D$ with Lipschitz bound $M$ and a rotation $R$ of $\rr^n$ such that $R(D)=\Omega$.  We observe that we can use Theorem~\ref{defT} to define  an extension operator  $T$  from $W^{l,p}(\Omega)$ to $W^{l,p}(\rr^n)$. 
	Indeed, if $T_{D}$ denotes the extension operator provided by Theorem~\ref{defT} for the special Lipschitz domain  $D$, then 
	it suffices to set  $Tf=(T_{D}(f\circ R))\circ R^{-1}$  
	for all  $f \in W^{l,p}(\Omega)$, and   it is easy to verify  that $T$ is a linear continuous extension operator from $W^{l,p}(\Omega)$ to $W^{l,p}(\rr^n)$ the norm of which depends only on $l,n, M$.
\end{remark}

\section{Stein's operator on general Lipschitz domains}

In this section we consider the case of Lipschitz domains of general type. In \cite{stein} they are called domains with minimally smooth boundary, and they are defined as follows. Recall that by domain we mean a connected open set. 

\begin{definition}\label{minsmooth}
	Given a domain $\Omega$  in $\rr^n$  we say that the boundary  $\partial \Omega$ is minimally smooth if there exist  $\varepsilon >0$, $N \in \nn$, $M>0$ and a sequence $\{U_i\}_{i=1}^s$ (where $s$ can be $+\infty$) of open sets such that:
	\begin{enumerate}[i)]
		\item if $x \in \partial \Omega,$ then $B_\varepsilon(x) \subset U_i,$ for some $i$, where $B_\varepsilon(x)$ is the open ball centred in $x$ of radius $\varepsilon.$
		\item No point of $\rr^n$ is contained in more than $N$ elements of the family $\{U_i\}_{i=1}^s.$
		\item For every $i=1,...,s$ there exist a special  Lipschitz domain $D_i$ and a rotation $R_i$ of $\rr^n$ such that
		\[ U_i\cap \Omega = U_i \cap R_i(D_i).\]
		\item The Lipschitz bound of $D_i$ does not exceed $M$ for every $i$.
	\end{enumerate}
	In this case,  we also say\footnote{note that this extra terminology is not present in \cite{stein} and is introduced here for our convenience} that $\Omega$ is a domain with  minimally smooth boundary and parameters  $\varepsilon $, $N$, $M$, $\{U_i\}_{i=1}^s$. 
\end{definition}

We now give the outline of the construction of the Stein extension operator for a domain with minimally smooth boundary. The details of this construction and the proof of Theorem \ref{Eteor} can be found in \cite{stein}. In the sequel,  given a set $U$ in $\rr^n$ and  $\varepsilon >0$ we set $U_\varepsilon=\{ x 	\in U \ | \ B_\varepsilon(x) \subset U\}$. 

Let $\Omega$ be a domain in $\rr^n$ with minimally smooth boundary and parameters   $\varepsilon,N,M, \{U_i\}_{i=1}^s$. We can construct a sequence of real-valued functions $\{\lambda_i\}_{i=1}^s$ defined in $\rr^n$, such that  for every  $i=1,...,s$ we have $\supp \lambda_i \subset U_i$,  $-1\le \lambda_i\le 1$, $\lambda_i(x)=1$ for all $x \in 
U_i{_{\varepsilon/2}}$,  $\lambda_i$ is of class $C^\infty$, has bounded derivatives of all orders and the bounds of the derivatives of $\lambda_i$ can be taken to be independent of $i.$
We can also construct two real-valued functions $\Lambda_+,\Lambda_-$ defined in $\rr^n$,  such that 
$\supp \Lambda_+ \subset   \{ x \in \Omega \ | \ d(x,\partial \Omega)\le \varepsilon \}\cup \{ x \in \rr^n \ | \ d(x,\partial \Omega)\le \varepsilon /2\},$
$\supp \Lambda_- \subset \Omega,$
$|\Lambda_+|,|\Lambda_-|\le 1$
$\Lambda_++\Lambda_- =1 $ in $\bar \Omega,$
$\Lambda_+,\Lambda_-$ are of class $C^\infty(\rr^n)$ with bounded derivatives of all orders.

Consider now the extension operators $T_i:W^{l,p}(R_i(D_i))\to W^{l,p}(\rr^n)$, defined as in Remark \ref{rotlip}. We define the extension operator $T$ for $\Omega$ as follows
\begin{equation} 
Tf(x):= \Lambda_+(x)\frac{\sum_{i=1}^s\lambda_i(x)T_i(\lambda_if)(x)}{\sum_{i=1}^s \lambda^2_i(x)}+\Lambda_-(x)f(x). \label{defEE}
\end{equation}

Then we have the  following important theorem  proved in \cite{stein}. 

\begin{theorem}[Stein's Extension Theorem - general case]\label{Eteor}
	Let $1\le p\le \infty, l,n \in \nn$. Let $\Omega$ be a domain in $\rr^n$ having minimally smooth boundary. Then the operator $T$ defined in \eqref{defEE} is a linear continuous operator from $W^{l,p}(\Omega)$ to $W^{l,p}(\rr^n).$
\end{theorem}

In order to prove that Stein's operator 
preserves Sobolev-Morrey spaces also in the general case,  we need to assume that the covering $\{ U_i\}_{i=1}^s$ in Definition~\ref{minsmooth} is a little more regular.  For this reason, we introduce the following natural definition.   

\begin{definition}
	Let $\Omega$ be a domain in $\rr^n$ with minimally smooth boundary and parameters $\varepsilon,M,N$, $\{ U_i\}_{i=1}^s.$ We say that $\{ U_i\}_{i=1}^s$ is a \textit{regular covering} for $\Omega$ if  for every $i=1,...,s.$, the open set   $U_i$ has the $\varepsilon$-ball property, i.e.,  if for every $ x\in U_i $ there exists an open ball $B$ of radius $\varepsilon$ contained in  $U_i$ such that $x \in B$. 
\end{definition}

The following lemma shows that   using regular coverings is not restrictive. 

\begin{lemma}
	Every domain in $\rr^n $ with minimally smooth boundary admits a regular covering.
\end{lemma}
\begin{proof}
	Let $\Omega$ be a domain in $\rr^n$ with minimally smooth boundary and parameters $\varepsilon,M,N, \{ U_i\}_{i=1}^s.$ Let 
	\[V_i:=\bigcup_{\substack{x \in \partial\Omega, \\ B_\varepsilon(x) \subset U_i}}B_\varepsilon(x) \]
	and consider the family $\{ V_i\}_{i=1}^{\widetilde s}$ containing the sets $V_i$ that are non-empty. Clearly $V_i$ has the $\varepsilon$-ball property for every $i=1,...,\widetilde s$. Moreover, it is immediate to verify that conditions i), ii), iii) and iv) in  Definition \ref{minsmooth} are satisfied with   $\{ U_i\}_{i=1}^s$ replaced by  $\{ V_i\}_{i=1}^{\widetilde s}$  and with the same constants $\varepsilon,M,N$. \end{proof}

Finally, we  prove that  operator $T$ defined in  \eqref{defEE}  using a regular covering preserves Sobolev-Morrey spaces. 

\begin{theorem}\label{mains}
	
	Let $1\le p<\infty $, $l\in {\mathbb{N}}$ and $\phi$ a function from $\rr^+$ to $\rr^+$. Let $\Omega$ be a domain in $\rr^n$ with minimally smooth boundary  and  parameters $\varepsilon,M,N$, $\{ U_i\}_{i=1}^s$, where  $\{ U_i\}_{i=1}^s$ is a regular covering for $\Omega$.  Let    $T$ be the operator defined in \eqref{defEE} using $\{ U_i\}_{i=1}^s$. Then 
	for every $\delta >0$ there exists $C>0$ such that  estimate \eqref{Sbound}  holds
	for all $f \in W^{l,p}(\Omega)$ and  $\alpha \in \nn_0^n$ with $|\alpha|\le l$.  Constant $C$ depends only on   $n, \varepsilon, l, M,N, \delta $ and on 
	the $L^{\infty}$-norms of the derivatives up to order $l$ of the  functions $\lambda_i$, $i=1,\dots , s$, $\Lambda^+$, $\Lambda^{-}$ appearing  in \eqref{defEE}. Moreover, if in addition $\Omega$ is bounded then $C$ 
	can be taken to be independent of $\delta$. 
\end{theorem}

\begin{proof}
	Let $\delta >0$ and let  $B$ an open ball in $\rr^n$  of radius $r$ with  $0<r<\delta$.  Let  $J=\{i \in \{1,...,s\} \ | \ B\cap U_i \neq \emptyset\}$. 
	
	{\bf Claim.} The cardinality  $\#J$ of $J$ satisfies $\#J\le \xi$, where $\xi$ is a constant  depending only on $n, \varepsilon,N,\delta$; moreover, if in addition $\Omega$ is bounded $\xi$ is independent of $\delta$. 
	
	We consider first the case when $\Omega$ is bounded. Then also its $\varepsilon$-neighbour\-hood $\Omega^\varepsilon=\{x \in \rr^n \,|\, d(x,\Omega)<\varepsilon\}$ is bounded. Moreover, by definition $U_i\cap \Omega^\varepsilon$ contains a ball of radius $\varepsilon$, hence $|U_i\cap \Omega^\varepsilon|>\varepsilon^n\omega_n,$ where $\omega_n$ is the volume of the $n$-dimensional unit ball. Since the covering $\{U_i\}_{i=1}^s$ has multiplicity less than $N$ and $U_i\cap \Omega^\varepsilon \subset\Omega^\varepsilon$, we have that $\sum_{i=1}^s|U_i\cap \Omega^\varepsilon|\le N |\Omega^\varepsilon|$. This implies that $s\le N |\Omega^\varepsilon|/(\varepsilon^n\omega_n)$, hence in particular $\#J\le N |\Omega^\varepsilon|/(\varepsilon^n\omega_n)=\xi.$ We observe that in this case $\xi$ 
		does not
	depend on $\delta.$
	
	We consider now the case when $\Omega$ is unbounded. Since the diameter of $B$ is less than $2\delta$, by Lemma \ref{covering} there exists a family of $m$ balls centered in $B$ of radius $\varepsilon$ that covers $B$, where $m$ depends only on $\delta,\varepsilon$ and $n$. Suppose now that $\#J>mp$, for some integer $p\in \nn$. Then at least one of these balls has non-empty intersection with at least $p+1$ elements of the family  $\{U_i\}_{i=1}^s$. Let's call this ball $B_\varepsilon$ and denote by $c_{B_\varepsilon}$ its center. Thus there exist points $x_i$, $i=1,...,p+1$, with $x_i \in B_\varepsilon \cap U_i.$ Since each $U_i$ has the $\varepsilon$-ball property, there are $B_i$, $i=1,...,p+1$, open balls of radius $\varepsilon$ with $B_i \subset U_i$ and $x_i \in B_i.$ We denote by  $c_i$ the centre of the ball $B_i$ and we notice that the set $\{c_1,...,c_{p+1}\}$ is contained in the ball of center $c_{B_\varepsilon}$ of radius $2\varepsilon.$ Indeed $|x_i-c_i|\le \varepsilon$ and $x_i \in B_\varepsilon,$ for every $i.$ Therefore by Lemma \ref{covering} we can cover the set $\{c_1,...,c_{p+1}\}$ with $q$ open balls of radius $\varepsilon/2$, where $q$ depends only on $n.$ Now suppose that $p>qN$, then at least one of these balls, that we label $B_{\varepsilon/2}$, contains at least $N+1$ points of the set $\{c_1,...,c_{p+1}\}.$ Without loss of generality we can suppose that they are $c_1,...,c_{N+1}$.  Then we must have that $B_1\cap B_2 \cap ... \cap B_{N+1}\neq \emptyset$ because each of these balls contains the center of $B_{\varepsilon/2}.$ However, since $B_i \subset U_i$ this is in contrast with property ii) of Definition \ref{minsmooth}. Thus,  if $\#J\ge mp$ then $p\le qN,$ hence $\#J<m(Nq+1) $ and the claim is proved.  
	
	We remark that the the value of the constant $C$ that will appear along the rest of the proof may vary, but it will remain dependent only on:  $n,M,l $ and on the $L^{\infty}$-norms of the derivatives up to order $l$ of the  functions $\lambda_i$, $i=1,\dots , s$, $\Lambda^+$, $\Lambda^{-}$.
	
	We can proceed with the proof of the theorem in the case $|\alpha|=0.$ Let $f \in W^{l,p}(\Omega)$.  By  applying the definition of $Tf$  we get
	
	\begin{align*}
	\left( \frac{1}{\phi(r)} \int_B |Tf(x)|^pdx \right)^\frac{1}{p} \le & \left( \frac{1}{\phi(r)} \int_{ B  \cap \Omega^c} \left|\Lambda_+(x) \frac{\sum_{i=1}^s \lambda_i(x)T_i(f\lambda_i)(x)}{\sum_{i=1}^s \lambda_i^2(x)}\right|^pdx \right)^\frac{1}{p} \\
	&\ \  \qquad\qquad\qquad\quad+\left( \frac{1}{\phi(r)} \int_{B \cap \Omega} | f(x) |^pdx \right)^\frac{1}{p}.
	\end{align*}
	The second integral can be estimated as follows
	\begin{equation}\label{ezbound}
	\phi(r)^{-1/p}\| f\|_{L^p(B\cap\Omega)}   \le\sum_{j=1}^m  \phi(r)^{-1/p}\| f\|_{L^p(B_j\cap\Omega)}  \le m\|f\|_{M_p^{\phi,\delta}(\Omega)} 
	\end{equation}
	where $B_1,...,B_m$ is a collection of balls of radius $r<\delta$ and centers in $\Omega$ with $m$ depending only on $n.$ To estimate the first integral we will use that $\sum_{i=1}^s \lambda_i^2\ge 1$ on  $\supp \Lambda_+ \cap \Omega^c$ and that $\supp \lambda_i \subset U_i$ for all $i=1,\dots ,s$. Moreover, we recall that exist  rotations $R_i$ and special Lipschitz domains $D_i$ such that $U_i\cap \Omega=U_i\cap R_i(D_i)$. We have
	
	\begin{align*}\lefteqn{
		\left( \frac{1}{\phi(r)} \int_{ B  \cap \Omega^c} \left|\Lambda_+(x) \frac{\sum_{i=1}^s \lambda_i(x)T_i(f\lambda_i)(x)}{\sum_{i=1}^s \lambda_i^2(x)}\right|^pdx \right)^\frac{1}{p} }\\
	&\qquad \le  \sum_{i\in J}  \left( \frac{1}{\phi(r)} \int_{ B  \cap \Omega^c} |T_i(f\lambda_i)(x)|^pdx \right)^\frac{1}{p} \le \sum_{i\in J} \| T_i(f\lambda_i)\|_{M_p^{\phi,\delta}(\rr^n)} \\
	&\qquad  \le C\sum_{i\in J} \| f\lambda_i\|_{M_p^{\phi,\delta}(R_i(D_i))} \le C \sum_{i\in J} \| f\|_{M_p^{\phi,\delta}(R_i(D_i)\cap U_i)}\\
	&\qquad =C\sum_{i\in J} \| f\|_{M_p^{\phi,\delta}(\Omega \cap U_i)} \le C\xi\| f\|_{M_p^{\phi,\delta}(\Omega )}.
	\end{align*}
	Here we have used inequality \eqref{Sbound} for $T_i$. This combined with \eqref{ezbound} implies the validity of \eqref{Sbound} when $|\alpha|=0.$
	
	We prove now \eqref{Sbound} when $|\alpha|>0.$  By the Leibniz rule we have  that for all $x \in B$
	\[ |D^\alpha Tf(x)|\le C\sum_{i \in J} \sum_{\beta\le \alpha} |D^\beta T_i(f\lambda_i)(x)| \chi_{ \Omega^c}(x)+ C\sum_{\beta\le \alpha} |D^\beta f(x)|\chi_{\Omega}(x)  \]
	where $C$ is a positive constant depending only on $\alpha,n$ and on the upper  bound of the derivatives  up to order $|\alpha|$ of the functions $\lambda_i$, $i=1,\dots , s$, $\Lambda^+$, $\Lambda^{-}$. 
	Hence 
	\begin{align*}
	\phi(r)^{-1/p}\| D^\alpha Tf\|_{L^p(B\cap\Omega)}  &\le C\sum_{i\in J}\sum_{\beta\le \alpha}  \phi(r)^{-1/p}\| D^\beta T_i(f\lambda_i)\|_{L^p(B\cap\Omega^c)}  \\
	&+ C\sum_{\beta\le \alpha} \phi(r)^{-1/p}\| D^\beta f\|_{L^p(B\cap\Omega)}.
	\end{align*}
	Arguing as before we can estimate the second  term  as follows
	\begin{equation}
	\sum_{\beta\le \alpha}  \phi(r)^{-1/p}\|D^\beta f\|_{L^p(B\cap\Omega)}  \le  m\sum_{\beta\le \alpha}  \| D^\beta f\|_{M_p^{\phi,\delta}(\Omega)}. \label{ezbound2}
	\end{equation}
	We can  also estimate the first  term  using inequality \eqref{Sbound} for $T_i$. In particular we get
	\begin{align*}  
	&\sum_{i\in J}\sum_{\beta\le \alpha}  \left( \frac{1}{\phi(r)} \int_{ B \cap \Omega^c} |D^\beta T_i(f\lambda_i)(x)|^pdx \right)^\frac{1}{p} \\ 
	&  \le C\sum_{i\in J}\sum_{\substack{\beta\le \alpha \\   |\gamma|\le |\beta|}}           \| D^\gamma (\lambda_if)\|_{M_p^{\phi,\delta}(R_i(D_i))} 
	\le C\sum_{i\in J}\sum_{\substack{\beta\le \alpha \\   |\gamma|\le |\beta|}} \| D^\gamma f\|_{M_p^{\phi,\delta}(R_i(D_i)\cap U_i)} \\
	& = C\sum_{i\in J}
	\sum_{\substack{\beta\le \alpha \\   |\gamma|\le |\beta|}}     \| D^\gamma f\|_{M_p^{\phi,\delta}(\Omega\cap U_i)}     \le C\sum_{i\in J} \sum_{|\beta|\le |\alpha|}\| D^\beta f\|_{M_p^{\phi,\delta}(\Omega)}   \\
	& \le C\xi\sum_{|\beta|\le |\alpha|}\| D^\beta f\|_{M_p^{\phi,\delta}(\Omega)}, \addtag \label{hardbound}
	\end{align*}   
	where $ C$ is a constant depending only on   $n,M,l $ and on the $L^{\infty}$-norms of the derivatives up to order $l$ of the  functions $\lambda_i$, $i=1,\dots , s$, $\Lambda^+$, $\Lambda^{-}$. Inequality \eqref{hardbound} together with \eqref{ezbound2} gives \eqref{Sbound} for $|\alpha|>0.$ We finally observe that in the proof of \eqref{Sbound} the only constant possibly depending on $\delta$ is $\xi$, but we  have also proved in the Claim above that  if $\Omega$ is bounded then  $\xi$ 
		does not 
	actually depend on $\delta$.  This completes the proof of the theorem. 
\end{proof}

\vspace{5pt}

{\bf Acknowledgments.} This paper represents a part of a dissertation written at the University of Padova by the second author under the guidance of the first author. 
 The  first author is also a member of the Gruppo Nazionale per l'Analisi Ma\-te\-ma\-ti\-ca, la Probabilit\`{a} e le loro Applicazioni (GNAMPA) of the
Istituto Nazionale di Alta Matematica (INdAM).
 This research was also supported by the INDAM - GNAMPA project 2017 ``Equazioni alle derivate parziali non lineari e disuguaglianze funzionali: aspetti geometrici ed analitici".

$ $ \vspace{4pt}

{\small

\noindent Pier Domenico  Lamberti\\
Department of Mathematics ``Tullio Levi-Civita''\\
University of Padova\\
Via Trieste 63\\
I-35121 Padova, Italy\\
E-mail address: lamberti@math.unipd.it\\

\noindent Ivan Yuri Violo \\
SISSA - Scuola Internazionale Superiore di Studi Avanzati\\
Via Bonomea 265\\ 
I-34136 Trieste, Italy\\
E-mail address: iviolo@sissa.it \\

}

\end{document}